\documentclass[11pt,reqno, final]{amsart}

\usepackage{amsmath,amssymb,amsthm,bbm,bm}
\usepackage{color}
\usepackage[colorlinks=true, allcolors=blue]{hyperref}
\usepackage{mathrsfs}
\usepackage{mathtools}

\usepackage[noabbrev,capitalize,nameinlink]{cleveref}
\crefname{equation}{}{}
\usepackage[noadjust]{cite}

\usepackage{fullpage}

\usepackage{graphics}
\usepackage{pifont}
\usepackage[T1]{fontenc}
\usepackage{tikz}
\usetikzlibrary{arrows.meta}
\usepackage{environ}
\usepackage{framed}
\usepackage{url}
\usepackage[linesnumbered,ruled,vlined]{algorithm2e}
\usepackage[noend]{algpseudocode}
\usepackage[labelfont=bf]{caption}
\usepackage{cite}
\usepackage{framed}
\usepackage[framemethod=tikz]{mdframed}
\usepackage{appendix}
\usepackage{graphicx}
\usepackage[textsize=tiny]{todonotes}
\usepackage{tcolorbox}
\usepackage{enumerate}
\allowdisplaybreaks[1]
\usepackage{enumerate}

\crefname{algocf}{Algorithm}{Algorithms}

\crefname{equation}{}{} 
\AtBeginEnvironment{appendices}{\crefalias{section}{appendix}} 

\usepackage[color]{showkeys} 
\colorlet{refkey}{orange!20}
\colorlet{labelkey}{blue!30}

\crefname{algocf}{Algorithm}{Algorithms}

\numberwithin{equation}{section}
\newtheorem{theorem}{Theorem}[section]
\newtheorem{proposition}[theorem]{Proposition}
\newtheorem{lemma}[theorem]{Lemma}

\crefname{claim}{Claim}{Claims}

\newtheorem*{question*}{Question}

\theoremstyle{definition}

\newtheorem*{definition*}{Definition}

\theoremstyle{remark}
\newtheorem*{remark}{Remark}

\newcommand{\snorm}[1]{\lVert#1\rVert}

\newcommand{\mb}{\mathbb}

\newcommand{\mbm}{\mathbbm}
\newcommand{\mc}{\mathcal}
\newcommand{\mf}{\mathfrak}
\newcommand{\mr}{\mathrm}

\newcommand{\on}{\operatorname}

\renewcommand{\epsilon}{\varepsilon}

\allowdisplaybreaks

\title{Optimal minimization of the covariance loss}
\author[Jain]{Vishesh Jain}
\address{Department of Statistics, Stanford University}
\email{visheshj@stanford.edu}

\author[Sah]{Ashwin Sah}
\author[Sawhney]{Mehtaab Sawhney}

\address{Department of Mathematics, Massachusetts Institute of Technology}
\email{\{asah,msawhney\}@mit.edu}

\thanks{Sah and Sawhney were supported by NSF Graduate Research Fellowship Program DGE-1745302. Sah was supported by the PD Soros Fellowship. }

\begin{document}

\begin{abstract}
Let $X$ be a random vector valued in $\mathbb{R}^{m}$ such that $\|X\|_{2} \le 1$ almost surely. For every $k\ge 3$, we show that there exists a sigma algebra $\mathcal{F}$ generated by a partition of $\mathbb{R}^{m}$ into $k$ sets such that
    \[\|\operatorname{Cov}(X) - \operatorname{Cov}(\mb{E}[X\mid\mathcal{F}])
    \|_{\mathrm{F}} \lesssim \frac{1}{\sqrt{\log{k}}}.\]
This is optimal up to the implicit constant and improves on a previous bound due to Boedihardjo, Strohmer, and Vershynin.

Our proof provides an efficient algorithm for constructing $\mc{F}$ and leads to improved accuracy guarantees for $k$-anonymous or differentially private synthetic data. We also establish a connection between the above problem of minimizing the covariance loss and the pinning lemma from statistical physics, providing an alternate (and much simpler) algorithmic proof in the important case when $X \in \{\pm 1\}^m/\sqrt{m}$ almost surely.  
\end{abstract}

\maketitle

\section{Introduction}\label{sec:intro}

Let $X$ be a random vector valued in $\mb{R}^{m}$. By slightly abusing notation, we identify $X$ with its law, which is a probability measure on $(\mb{R}^{m}, \mc{G})$, where $\mc{G}$ is a sigma-algebra on $\mb{R}^{m}$. Let $\mc{F}$ be a sigma sub-algebra of $\mc{G}$ and let $Y = \mb{E}[X \mid \mc{F}]$ denote the corresponding conditional expectation. In particular, $\mb{E}[X] = \mb{E}[Y]$. Let 
\[\Sigma_{X} := \mb{E}[(X-\mb{E}X)(X-\mb{E}X)^{T}]\]
denote the covariance matrix of $X$ and let $\Sigma_{Y}$ denote the covariance matrix of $Y$. When $m=1$, $\Sigma_X$ is precisely the variance of $X$, which we denote by $\on{Var}(X)$, and similarly for $\Sigma_Y$. The familiar law of total variance asserts that
\[\on{Var}(X) - \on{Var}(Y) = \mb{E}(X-Y)^{2} \geq 0,\]
so that taking a conditional expectation results in a loss of variance. This phenomenon extends to higher dimensions as the law of total covariance:   
\begin{align}
\label{eqn:law-covariance}
    \Sigma_{X} - \Sigma_{Y} = \mb{E}(X-Y)(X-Y)^{T} \succeq 0,
\end{align}
where $\succeq$ denotes the usual Loewner order on positive semi-definite matrices. 

Recently, motivated by the design of privacy-preserving synthetic data (see the discussion in \cref{sec:application}), Boedihardjo, Strohmer, and Vershynin \cite{BSV21} asked the following fundamental question: \emph{how much covariance is lost upon taking a conditional expectation?} The answer to this clearly depends on the sigma sub-algebra $\mc{F}$ (for instance, the choice $\mc{F} = \mc{G}$ loses no covariance, whereas the trivial sigma sub-algebra $\mc{F} = \{\emptyset, \mb{R}^{m}\}$ leads to the maximum possible covariance loss of $\Sigma_X$). This suggests restricting the `complexity' of the sigma sub-algebra $\mc{F}$ and investigating how much covariance is necessarily lost upon taking a conditional expectation with respect to a sigma sub-algebra $\mc{F}$ with a given complexity. Moreover, for applications, one would like to be able to find the best possible (at least asymptotically) sigma sub-algebra with a given complexity in an efficient manner. 

Since every finitely generated sigma-algebra $\mc{F}$ may be viewed as the sigma-algebra generated by a partition of $\mb{R}^{m}$ into $k$ sets (for some finite $k$), a natural and useful measure of complexity of $\mc{F}$ is the number of sets in the underlying partition, $k$. With this notion of complexity, and measuring covariance loss in the Frobenius norm, Boedihardjo, Strohmer, and Vershynin \cite[Theorem~1.2]{BSV21} showed that there exists an absolute constant $C > 0$ such that for any random vector $X$ valued in $\mb{R}^{m}$ for which $\|X\|_{2} \leq 1$ almost surely, and for every $k \geq 3$, there exists a partition of $\mb{R}^{m}$ into at most $k$ sets such that for the sigma-algebra $\mc{F}$ generated by this partition, $Y = \mb{E}[X \mid \mc{F}]$ satisfies the dimension-independent bound
\begin{align}
\label{eqn:old-bound}
    \|\Sigma_X - \Sigma_Y\|_{\mr{F}} \leq C\sqrt{\frac{\log\log{k}}{\log k}},
\end{align}
where for $A \in \mb{R}^{m\times m}$, $\|A\|_{\mr{F}} = \sqrt{\sum_{i,j}A_{ij}^{2}}$ denotes its Frobenius norm. They noted \cite[Proposition~3.14]{BSV21} that the upper bound is optimal up to the factor of $\sqrt{\log\log{k}}$.

Note that in the case when $X$ is the uniform distribution over $x_1,\dots, x_n \in \mb{R}^{m}$ with $\max_{i}\|x_i\|_{2}\leq 1$, and $\mc{F}$ is generated by a partition into $k$ sets, the dimension-independence of \cref{eqn:old-bound} stands in stark-contrast to (a variation of) the $k$-means objective 
\begin{align*}
\inf_{y_1,\dots, y_k \in \mb{R}^{m}, I_1 \sqcup \dots \sqcup I_{k} = [n]}\sum_{i=1}^{k}\sum_{j \in I_i}\|x_j - y_i\|_{2},
\end{align*}
which bounds $\inf_{\mc{F}}\|\Sigma_{X} - \Sigma_{Y}\|_{\mr{F}}$ from above (via a direct application of Jensen's inequality) and, in general, can decay as slowly as $\Omega(k^{-1/m})$, which is significantly worse in the high-dimensional regime of interest here.

As our main result, we remove the gap between the upper bound in \cref{thm:main} and the lower bound in \cite[Proposition~3.14]{BSV21}, thereby obtaining an optimal and algorithmic answer to the problem of minimizing covariance loss raised by Boedihardjo, Strohmer, and Vershynin.     

\begin{theorem}\label{thm:main}
Let $X$ be a random vector valued in $\mb{R}^m$ which satisfies $\snorm{X}_2\le 1$ almost surely. Then for every $k\ge 3$, there exists a partition of $\mb{R}^{m}$ into at most $k$ sets such that for the associated $\sigma$-algebra $\mc{F}$, the conditional expectation $Y = \mb{E}[X\mid \mc{F}]$ satisfies 
\[\snorm{\Sigma_{X} - \Sigma_{Y}}_{\mr{F}}\leq \frac{C}{\sqrt{\log k}},\]
where $C$ is an absolute constant. 
\end{theorem}
As noted earlier, our bound is optimal up to the value of the absolute constant $C$. We prove \cref{thm:main} in \cref{sec:proof-general}. Before doing so, in \cref{sec:proof-boolean}, we provide a completely different proof of \cref{thm:main} in the case when $X \in \frac{1}{\sqrt{m}}\cdot \{\pm 1\}^{m}$ based on the pinning lemma from statistical physics; this case is especially important for applications, since it corresponds to the case of Boolean `true' data in the setting of \cref{sec:application}. The proof in \cref{sec:proof-boolean} is much simpler than the general proof in \cref{sec:proof-general} and provides a significantly faster and simpler algorithm for finding $\mc{F}$.  
\begin{remark}
\label{rmk:moreover}
By following exactly the same procedure as in \cite[Section~3.6]{BSV21}, if the probability space has no atoms, then the partition can be made with exactly $k$ sets, all of which have the same probability $1/k$.
\end{remark}

\begin{remark}
By combining \cref{thm:main} with the tensorization principle \cite[Theorem~3.10]{BSV21}, we immediately obtain an analog of \cref{thm:main} for higher moments, which improves \cite[Corollary~3.12]{BSV21} by a factor of $\sqrt{\log\log{k}}$: for all $d\geq 2$,
\begin{align}
\label{eq:higher-order}
\|\mb{E}X^{\otimes d} - \mb{E}Y^{\otimes d}\|_{\mr{F}} \leq 4^{d}\cdot \frac{C}{\sqrt{\log{k}}},
\end{align}
where $C$ is the absolute constant appearing in \cref{thm:main}. Here, $X^{\otimes d} \in \mb{R}^{m \times m \dots \times m}$ is defined by $X^{\otimes d}(i_1,\dots, i_d) := X(i_1)\cdots X(i_d)$, where $i_1,\dots, i_d \in [m]$ (and similarly for $Y^{\otimes d}$), and for $A \in \mb{R}^{m\times \dots m}$, $\|A\|_{\mr{F}} := \sqrt{\sum_{i_1,\dots, i_d \in [m]}A(i_1,\dots, i_d)^{2}}$. 
\end{remark}

\subsection{Applications to the design of privacy-preserving synthetic data}
\label{sec:application}
As mentioned earlier, the problem of minimizing covariance loss was studied in \cite{BSV21} with a view towards designing privacy-preserving synthetic data. Here, one is given `true' data points $x_1,\dots, x_n \in \mb{R}^{m}$ and would like to construct a map $\mc{A} : \{x_1,\dots, x_n\} \to \mb{R}^{m}$ such that the set of `synthetic' data $\{\mc{A}(x_1),\dots, \mc{A}(x_n)\}$ is both `private' and `accurate'. We refer the reader to \cite{BSV21} for a much more detailed discussion of these notions and further references, limiting ourselves here to the most basic application of \cref{thm:main}.

A popular notion of preserving privacy is $k$-anonymity \cite{S02}; for synthetic data, this is the requirement that for any $y \in \{\mc{A}(x_1),\dots, \mc{A}(x_n)\}$, the preimage $\mc{A}^{-1}(y)$ has cardinality at least $k$. In words, the true data is transformed into synthetic data in such a manner that the information of each person in the dataset cannot be distinguished from that of at least $k-1$ other individuals in the dataset.

Let us quickly discuss how \cref{thm:main} may be used to obtain accurate $\lfloor n/k \rfloor$-anonymous synthetic data. Given true data $x_1,\dots, x_n \in \mb{R}^{m}$, we consider the random vector $X$ which takes on each value $x_{i}$ with probability $1/n$ each. Given $k\geq 3$, \cref{thm:main} gives a partition of $\mb{R}^{m}$ into $k$ sets, which induces a partition $[n] = I_1 \cup \dots \cup I_{k}$ and a sigma algebra $\mc{F}$ on $\{x_1,\dots, x_n\}$. Moreover, by a slight variation of the remark following \cref{thm:main}, we may assume that $|I_i| \geq \lfloor n/k \rfloor$ for all $i \in [k]$. For $j \in [n]$, let $I(j) \in \{I_1,\dots, I_k\}$ denote the unique subset of $[n]$ such that $j \in I(j)$. Then, the conditional expectation $Y = \mb{E}[X \mid \mc{F}]$ corresponds to the synthetic data map 
\[x_{j} \mapsto y_{I(j)} := \frac{1}{|I(j)|}\sum_{i \in I(j)}x_i.\]
This map is $\lfloor n/k \rfloor$-anonymous, by construction. As for accuracy, it follows from  \cref{thm:main} that, with $Y$ the random vector which takes on each value $y_\ell$ with probability $1/k$, 
\[\|\Sigma_{X} - \Sigma_{Y} \|_{\mr{F}} \lesssim \frac{1}{\sqrt{\log{k}}},\]
so that the synthetic data is accurate in the sense that it approximately preserves, on average, the second order marginals of the true data. This can be extended to higher-order marginals using \cref{eq:higher-order}. 

The above idea is adapted in \cite{BSV21} to extract additional guarantees for anonymous, synthetic data (see \cite[Theorems~4.4, 4.6]{BSV21}). In both cases, replacing \cref{eqn:old-bound} with our \cref{thm:main} leads to quantitative improvements by a factor of $\log\log{k}$.

Finally, we remark that in \cite[Theorems~5.9-5.11]{BSV21}, a generalization of \cref{eqn:old-bound} is used with additional arguments to design differentially-private synthetic data. Our proof of \cref{thm:main} in \cref{sec:proof-general} can also be generalized using similar arguments as in \cite{BSV21} to yield versions of \cite[Theorems~5.9-5.11]{BSV21} without the $\log\log{n}$ factor there; we leave the details to the interested reader.

\section{Proof of \texorpdfstring{\cref{thm:main}}{Theorem 1.1} for Boolean Data}
\label{sec:proof-boolean}

In this section, we provide a proof of \cref{thm:main} in the case when $X$ is valued in $\{\pm 1\}^{m}/\sqrt{m}$ almost surely. In the setting of \cref{sec:application}, this corresponds to the case when the true data is Boolean and hence is particularly relevant for applications. Our proof relies on the so-called pinning lemma from statistical physics, discovered independently by Montanari \cite{Mon08} and by Raghavendra and Tan \cite{RT10}. The statement below follows by combining \cite[Lemma~4.5]{RT10} with Pinsker's inequality (cf.~the proofs of \cite[Lemmas~4.2, A.2]{JKR19}).

\begin{lemma}
\label{lem:pinning}
Let $X_1,\dots, X_m$ be a collection of $\{\pm 1\}$-valued random variables. Then, for any $\ell \in [m]$, we have that
\[\mb{E}_{t \sim \{0,1,\dots,\ell\}}\mb{E}_{S\sim \binom{[m]}{t}}\left[\mb{E}_{X_S}\left(\sum_{i\neq j\in [m]}\on{Cov}(X_i, X_j \mid X_S)^{2}\right)\right] \leq \frac{8m^2\log 2}{\ell}.\]
\end{lemma}

Roughly speaking, the intuition behind the pinning lemma is the following: either the average (pairwise) covariance between the random variables $X_1,\dots, X_{m}$ is already small (in which case, we're done) or the average covariance is not small. In the latter case, we expect a random coordinate $X_i$ to contain substantial information about many of the other coordinates $X_1,\dots, X_{m}$, so that conditioning on a small random subset of the coordinates makes the average conditional covariance sufficiently small.  

Given \cref{lem:pinning}, we can quickly deduce \cref{thm:main} for Boolean data.

\begin{proof}[Proof of \cref{thm:main} for Boolean data]
Recall that $X$ is valued in $\{\pm 1\}^{m}/\sqrt{m}$ almost surely. Note that we may assume that $m\ge \log_2{k}$; otherwise $X$ takes on at most $2^{m} \leq k$ values, so that the sigma algebra $\mc{F}$ generated by the partition of $\{\pm 1\}^{m}/\sqrt{m}$ which assigns each point to its own part has at most $k$ parts and satisfies $Y:= \mb{E}[X \mid \mc{F}] = X$. 

Now, let $t$ be chosen uniformly from $\{0,1,\dots, \log_{2}k\}$ and let $S$ be chosen uniformly from $\binom{[m]}{t}$. This provides a decomposition of $\{\pm 1\}^{m}/\sqrt{m}$ into at most $2^{t} \leq k$ clusters, where each cluster consists of all points of $\{\pm 1\}^{m}/\sqrt{m}$ which agree on the coordinates in $S$. In other words, each cluster corresponds to a setting of $X_S := (X_i)_{i\in S} \in \{\pm 1\}^{S}/\sqrt{m}$. Let $\mc{F}$ denote the sigma algebra generated by these clusters and let $Y = \mb{E}[X \mid \mc{F}] = \mb{E}[X \mid X_{S}]$.
Let $\Sigma_X$ and $\Sigma_Y$ denote the covariance matrices of $X$ and $Y$ respectively. Then,
\begin{align*}
    \mb{E}_S \|\Sigma_X - \Sigma_Y\|_{\mr{F}} 
    &= \mb{E}_S\| \mb{E}_X (X-\mb{E}[X \mid X_S])(X-\mb{E}[X \mid X_S])^{T}\|_{\mr{F}} &\textrm{(from \cref{eqn:law-covariance})}\\
    &\le \mb{E}_S \mb{E}_{X_S}\|\mb{E}_{X \mid X_S}(X-\mb{E}[X \mid X_S])(X-\mb{E}[X \mid X_S])^{T}\|_{\mr{F}} &\textrm{(norm convexity)}\\
    & = \mb{E}_S \mb{E}_{X_S}\sqrt{\sum_{i\neq j\in [m]} \on{Cov}(X_i,X_j \mid X_S)^2  + \sum_{i\in [m]} \on{Var}(X_i \mid X_S)^2}\\
    & \le \sqrt{\mb{E}_S \mb{E}_{X_S}\bigg(\sum_{i\neq j\in [m]} \on{Cov}(X_i,X_j \mid X_S)^2  + \sum_{i\in [m]} \on{Var}(X_i \mid X_S)^2\bigg)} & \textrm{(Jensen)}\\
    & \le \sqrt{\frac{8m^2\log 2}{\log_{2}k}\cdot \frac{1}{m^2}  + m\cdot \frac{1}{m^2}} \leq \frac{3}{\sqrt{\log_2 k}},
\end{align*}
where the first term in the penultimate inequality follows by applying \cref{lem:pinning} with $\ell = \log_{2}k$ and rescaling by a factor of $m^{-2}$ (since each $X_i$ is valued in $\{\pm 1\}/\sqrt{m}$) and the second term in the penultimate inequality follows by noting that $\on{Var}(X_i \mid X_S) \leq 1/m$ (again, since $X_i \in \{\pm 1\}/\sqrt{m}$).

Finally, by Markov's inequality,
\[\mb{P}_{S}\left[\|\Sigma_{X} - \Sigma_Y\|_{\mr{F}} \geq \frac{9}{\sqrt{\log_{2}k}}\right] \leq \frac{1}{3},\]
so we have a very simple randomized algorithm for finding (with probability at least $2/3$) a sigma algebra $\mc{F}$ obtaining the desired guarantee:  first choose $t$ uniformly from $\{0,1,\dots, \log_{2}k\}$, then choose $S$ uniformly from $\binom{[m]}{t}$, and finally decompose $\{\pm 1\}^{m}/\sqrt{m}$ based on the values of the coordinates in $S$. 
\end{proof}

\section{Proof of \texorpdfstring{\cref{thm:main}}{Theorem 1.1}}
\label{sec:proof-general}

In this section, we prove \cref{thm:main} for general random vectors $X \in \mb{R}^{m}$ satisfying $\|X\|_{2} \leq 1$ almost surely. As in \cite{BSV21}, we use principal component analysis to reduce to the case where $m = c\log{k}$, for a sufficiently small absolute constant $c > 0$. However, our treatment of the dimension-reduced problem is rather different from \cite{BSV21}. Indeed, whereas \cite{BSV21} partitions the dimension-reduced random vector according to the closest point in a volumetric epsilon-net (thereby, only exploiting the information that $\|X\|_{2} \leq 1$ almost surely), our clustering scheme also takes into account the distributional profile of the dimension-reduced random vector; briefly, we place each `heavy' point into its own cluster, place nearby points, which are `collectively light' into a single cluster, and for the intermediate case, adopt a randomized rounding scheme to cluster the points. In particular, our proof provides another instance where nets based on randomized rounding provide better control than volumetric nets (see \cite{Tik20,KL20,LTV21} for some other recent examples).   

This section is organized as follows: in \cref{sec:key-estimate}, we show how to appropriately cluster points in the most challenging `intermediate' case, mentioned above (\cref{prop:key-estimate}). Given this, the proof of \cref{thm:main} is completed in \cref{sec:finish-proof} by following the aforementioned decomposition into heavy, collectively light, and intermediate cases. 

\subsection{Key estimate}
\label{sec:key-estimate}
Let $p = c\log{k}$, where $c$ is a sufficiently small positive universal constant (for instance, $c\in(0,1/120)$ is certainly sufficient). Let 
\[\gamma := \frac{e^{-(\log k)/(4p)}}{\sqrt{p}} = \frac{e^{-1/(4c)}}{\sqrt{c\log{k}}}.\]
Let $X$ be a random vector valued in $x_{0} + [-\gamma/2, \gamma/2]^{p}$, supported on finitely many points, such that for any $x \in \on{supp}(X)  =: \mc{X}$, we have $\mb{P}[X = x]\leq k^{-1/3}$. Let $\mc{W} := x_{0} + \{\pm 3\gamma/2\}^{p}$. For each $x \in \mc{X}$, let $w_{x} \in \mc{W}$ be a random vector defined as follows: $\mb{E}[(w_{x})_{i}] = x_{i}$, and the random variables $(w_{x})_{i}$ are independent. In words, the vector $w_{x}$ is obtained by randomly rounding $x$ to a point in $\mc{W}$ so that $w_x$ has mean $x$; it is easily seen that such a distribution $w_x$ is unique. Moreover, for distinct $x \in \mc{X}$, the random vectors $w_{x}$ are independent.

Now, given a realisation of the random vectors $w_{x}$, for each $w \in \mc{W}$, let $$C_{w} := \{x \in \mc{X} : w_{x} = w\},$$
so that $C_{w}$ consists of those points in $\mc{X}$ which are rounded to $w$. Let $\mc{F}$ denote the sigma-algebra corresponding to the partition $(C_{w})_{w\in \mc{W}}$. Note that $\mc{F}$ is random, depending on the realisation of $w_{x}$. 

In our analysis, we will also require the following random vector, which should be viewed as an idealised version of $\mb{E}[X \mid \mc{F}]$; this random vector, which we denote by $Z$, takes on the value
\[z_{w} := \frac{\sum_{x \in \mc{X}} x \mb{P}[w_{x} = w]\mb{P}[X=x]}{\sum_{x \in \mc{X}}\mb{P}[w_{x} = w]\mb{P}[X = x]}\]
with probability
\[q_{w} := \sum_{x \in \mc{X}}\mb{P}[w_{x} = w]\mb{P}[X = x]\]
for each $w \in \mc{W}$. We begin with the following preliminary, but key, lemma.

\begin{lemma}
\label{lem:key-estimate}
With notation as above,
\[\|\Sigma_{X} - \Sigma_{Z}\|_{\mr{F}} = \|\mb{E}[XX^{T}] - \mb{E}[ZZ^{T}]\|_{\mr{F}} \leq \frac{36e^{-1/(2c)}}{\sqrt{c\log{k}}}\]
\end{lemma}
\begin{proof}
The first equality follows from the observation that $\mb{E}[Z] = \mb{E}[X]$. We proceed to prove the inequality. For convenience of notation, let
\[\mu_{x,w} := \mb{P}[X = x]\mb{P}[w_{x} = w]\]
We have
\begin{align*}
    \|\mb{E}[XX^{T}] &- \mb{E}[ZZ^{T}]\|_{\mr{F}} = \|\sum_{x,w}\mu_{x,w}(xx^{T}-z_{w}z_{w}^{T})\|_{\mr{F}} = \|\sum_{x,w}\mu_{x,w}(xx^{T}-xz_{w}^{T}-z_{w}x^{T}+z_{w}z_{w}^{T})\|_{\mr{F}}\\
    &= \|\sum_{x,w}\mu_{x,w}(x-z_{w})(x-z_{w})^{T}\|_{\mr{F}}\\
    &\leq^{(1)} 2\|\sum_{x,w}\mu_{x,w}(x - w)(x-w)^{T}\|_{\mr{F}} + 2\|\sum_{x,w}\mu_{x,w}(w-z_{w})(w-z_{w})^{T}\|_{\mr{F}}\\
    &\leq^{(2)} 2\|\sum_{x,w}\mu_{x,w}(x - w)(x-w)^{T}\|_{\mr{F}} + 2\|\sum_{x,w}\mu_{x,w}\sum_{x' \in \mc{X}}\frac{\mu_{x',w}}{q_{w}}(w-x')(w-x')^{T}\|_{\mr{F}}\\
    &= 2\|\sum_{x,w}\mu_{x,w}(x - w)(x-w)^{T}\|_{\mr{F}} + 2\|\sum_{w} \sum_{x' \in \mc{X}}\mu_{x',w}(w-x')(w-x')^{T}\|_{\mr{F}}\\
    &= 4\|\sum_{x,w} \mu_{x,w}(x-w)(x-w)^{T}\|_{\mr{F}}\\
    &\leq 4 \max_{x \in \mc{X}}\|\mb{E}_{w_x}[(x - w_{x})(x-w_{x})^{T}] \|_{\mr{F}}\\
    &\leq^{(3)} 4\max_{x \in \mc{X}}\| 9\gamma^{2}\cdot \on{Id}_{p\times p}\|_{\mr{F}}\\
    &\leq 36\gamma^{2}\sqrt{p} = \frac{36e^{-1/(2c)}}{\sqrt{c\log{k}}},
\end{align*}
as desired. 

Inequality (1) follows since $(x+y)(x+y)^{T} \preceq 2xx^{T} + 2yy^{T}$ for any vectors $x,y$ and $0 \preceq A \preceq B$ for symmetric matrices $A,B$ implies that $\|A\|_{\mr{F}} \leq \|B\|_{\mr{F}}$. (To see this inequality note that $\|B\|_{\mr{F}}^2-\|A\|_{\mr{F}}^2 = \on{tr}(B^2-A^2) = \on{tr}((B-A)(B+A)) = \on{tr}((B+A)^{1/2}(B-A)(B+A)^{1/2}) \ge 0$.) Inequality (2) follows since for any collection of vectors $y_{1},\dots, y_{m}$ and for any $p_{1} \geq 0,\dots, p_{m} \geq 0$ such that $\sum_{i}p_{i} = 1$, we have $(\sum_{i}p_i y_i)(\sum_{i}p_i y_i)^{T} \preceq \sum_{i} p_{i}y_i y_i^{T}$, as is verified by noting that for any vector $u$,
\begin{align*}
    u^{T}(\sum_{i}p_i y_i)(\sum_{i}p_i y_i)^{T}u
    &= \left(\sum_{i}p_i (u^{T}y_i)\right)^{2}\\
    &\leq \left(\sum_{i}p_{i}\right)\cdot \left(\sum_{i}p_{i}(u^{T}y_i)^{2}\right) &\textrm{(Cauchy--Schwarz)}\\
    &= u^{T}\left(\sum_{i}p_i y_{i}y_{i}^{T}\right)u.
\end{align*}
Finally, inequality (3) uses that $\mb{E}[(w_{x})_i] = x_{i}$, the independence of $(w_{x})_i$ and $(w_{x})_{j}$, and the crude estimate $|(x-w_{x})_{i}| \leq 3\gamma$. 
\end{proof}

The following is the main result of this subsection.

\begin{proposition}
\label{prop:key-estimate}
There exists an absolute constant $K > 0$ such that for all $k\geq K$, and with $Y = \mb{E}[X \mid \mc{F}]$ (notation as above), we have
\[\|\Sigma_{X} - \Sigma_{Y}\|_{\mr{F}} = \|\mb{E}[XX^{T}] - \mb{E}[YY^{T}]\|_{\mr{F}} \leq \frac{36e^{-1/(2c)}}{\sqrt{c\log{k}}} + k^{-1/48},\]
with probability (over the realisation of $\mc{F}$) at least $1-\exp(-k^{1/7})$. 
\end{proposition}

\begin{proof}
Without loss of generality we may assume that $x_0 = 0$. By \cref{lem:key-estimate} and the triangle inequality, it suffices to show that for all sufficiently large $k$, except with probability at most $\exp(-k^{1/7})$,
\[\|\mb{E}[ZZ^{T}] - \mb{E}[YY^{T}]\|_{\mr{F}} \leq k^{-1/48}.\]
For convenience of notation, for $w \in \mc{W}$ let
\[p_{w} := \sum_{x \in \mc{X}}\mb{P}[X = x]\mbm{1}[w_{x} = w]\]
and for $w \in \mc{W}$, $i \in [p]$, let
\[(y_{w})_i := \sum_{x\in \mc{X}}x_i \mb{P}[X=x]\mbm{1}[w_{x}=w].\]
By Hoeffding's inequality, for a given $w \in \mc{W}$,
\begin{align*}
    \mb{P}\left[|p_{w} - q_{w}| \geq k^{-1/12} \right] 
    &\leq \exp(-2k^{-1/6}/\sum_{x}\mb{P}[X=x]^{2})\\
    &\leq^{(1)} \exp(-2k^{-1/6}/k^{-1/3})\\
    & = \exp(-2k^{1/6}),
\end{align*}
where inequality (1) uses $\sum_{x}\mb{P}[X=x]^{2} \leq \max_{x}\mb{P}[X=x] \leq k^{-1/3}$, by assumption. Similarly, for a given $w \in \mc{W}$ and $i \in [p]$, we have 
\begin{align*}
    \mb{P}\left[|(y_{w})_i - q_{w}\cdot (z_w)_{i}| \geq \gamma \cdot k^{-1/12} \right] 
    &\leq \exp(-2\gamma^{2}k^{-1/6}/\sum_{x}x_{i}^{2}\mb{P}[X=x]^{2})\\
    & \leq \exp(-2k^{1/6}).
\end{align*}

Let $\mc{E}$ denote the event that $|\sum_{x \in \mc{X}}\mb{P}[X = x]\mbm{1}[w_{x} = w] - q_{w}| \leq k^{-1/12}$ and $|(y_{w})_i - q_{w}\cdot (z_w)_{i}| \leq \gamma \cdot k^{-1/12}$ for all $w \in \mc{W}, i \in [p]$. By the preceding discussion,
\[\mb{P}[\mc{E}^{c}] \leq 2\cdot 2^{p}\cdot p\cdot \exp(-2k^{1/6}) \leq \exp(-k^{1/7})\]
for all sufficiently large $k$.  Moreover, for every $i \in [p]$, $x \in \mc{X}$, and $\epsilon \in \{\pm 1\}$,  we have $\mb{P}[(w_{x})_i = \epsilon \cdot 3\gamma/2] \geq 1/3$, so that for every $w \in \mc{W}$,
\begin{align*}
    q_{w} \geq 3^{-p} \geq k^{-3c/2},
\end{align*}
and hence, on the event $\mc{E}$, we have for all ${w} \in \mc{W}$ that
\[p_{w} = q_{w} \pm k^{-1/12} = q_{w}(1 \pm k^{-1/24}),\]
assuming that $c < 1/36$. Finally, we see that on the event $\mc{E}$, 
\begin{align*}
    \|\mb{E}[ZZ^{T}] - \mb{E}[YY^{T}]\|_{\mr{F}}
    &= \|\sum_{w \in \mc{W}}(q_{w}z_{w}z_{w}^{T} - y_{w}y_{w}^{T}/p_{w})\|_{\mr{F}}\\
    &\leq \|\sum_{w} (q_{w} z_{w} - y_{w})z_{w}^{T} \|_{\mr{F}} + \|\sum_{w}y_{w}(z_{w}^{T} - y_{w}^{T}/p_{w})\|_{\mr{F}}\\
    &\leq 2^{p}\cdot (\max_{w}\|q_{w}z_{w}-y_{w}\|_{2}\|z_{w}\|_{2} + \max_{w}q_{w}^{-1}\|y_{w}\|_{2}\|q_{w}z_{w}^{T} - y_{w}\cdot q_{w}/p_{w}\|_{2})\\
    &\leq 2^{p}\left(\gamma^{2}p\cdot k^{-1/12} + k^{3c/2}\gamma\sqrt{p}\max_{w}(\|q_{w}z_{w}^{T} - y_{w}\|_{2} + \|y_{w}(1-q_{w}/p_{w})\|_{2})\right)\\
    &\leq \gamma^{2}pk^{c}\cdot k^{-1/12} + k^{5c/2}\gamma^{2}{p}\cdot k^{-1/12} + k^{5c/2}\gamma^{2}p\cdot \max_{w}|1-q_{w}/p_{w}|\\
    &\leq 3\cdot k^{5c/2}\gamma^{2}p\cdot k^{-1/24} \leq k^{-1/48},
\end{align*}
provided that $c < 1/120$.
\end{proof}

\subsection{Finishing the proof}
\label{sec:finish-proof}
With \cref{prop:key-estimate}, we are ready to prove \cref{thm:main} through a sequence of reductions. Recall that in the statement of \cref{thm:main}, $X$ is a random vector valued in $\mb{R}^{m}$ which satisfies $\|X\|_{2}\leq 1$ almost surely. Without loss of generality, we may assume that $X$ is finitely supported, by rounding the points in the support to a sufficiently fine $\epsilon$-net with respect to the Euclidean metric (see, e.g., \cite[Lemma~3.6]{BSV21}). 

Next, we show that it suffices to assume that $X$ is valued in $\mb{R}^{p}$, for $p = c\log{k}$, where $c$ is as in \cref{sec:key-estimate}. The following lemma is a slight modification of \cite[Lemmas~3.2,3.3]{BSV21}. 

\begin{lemma}
\label{lem:PCA}
Suppose that $X$ is a random vector with $\|X\|_{2} \leq 1$ almost surely. Let $S = \mb{E}[XX^T]$ and let $P$ the projection onto the subspace corresponding to the largest $t \geq 1$ eigenvectors of $S$. Let $Y = \mb{E}[X \mid PX]$. Then, 
\[\|\Sigma_{X} - \Sigma_Y\|_{\mr{F}} = \snorm{\mb{E}[XX^T]-\mb{E}[YY^T]}_{\mr{F}}\le \frac{1}{\sqrt{t}}.\]
\end{lemma}
\begin{proof}
The equality holds since $\mb{E}[X] = \mb{E}[Y]$. For the inequality, we note that, with $A := \mb{E}[XX^{T}] - \mb{E}[YY^{T}]$,
\begin{align*}
    \|A\|_{\mr{F}} &\leq^{(1)} \|PAP\|_{\mr{F}} + \|(I-P)\mb{E}[XX^T](I-P)\|_{\mr{F}}\\
    &\leq^{(2)}\|PAP\|_{\mr{F}} + \frac{1}{\sqrt{t}}\\
    &= \|\mb{E}(PX-PY)(PX-PY)^{T}\|_{\mr{F}} + \frac{1}{\sqrt{t}}\\
    &=^{(3)} \frac{1}{\sqrt{t}},
\end{align*}
where (1) follows from the proof of \cite[Lemma~3.2]{BSV21}, (2) follows from \cite[Lemma~3.3]{BSV21}, and (3) follows since $PY = P\mb{E}[X \mid PX] = PX$. 
\end{proof}
By taking $t = c\log{k}$ in \cref{lem:PCA} and using the triangle inequality, we see that it suffices to prove \cref{thm:main} for $X \in \mb{R}^{p}$, with $p = c\log{k}$ (the clustering in the original problem corresponds to applying the map $P^{-1}$ to the clustering in the dimension-reduced problem). Therefore, consider such an $X$, and recall that we may assume that $X$ is finitely supported, denoting the support by $\mc{X}$. Let
\[\mc{X}^{(1)} = \{x \in \mc{X} : \mb{P}[X=x] \geq 3/k\}.\]
Note that $|\mc{X}^{(1)}| \leq k/3$. By assigning each point in $\mc{X}^{(1)}$ to its own cluster, it suffices to find a clustering of the points in $\mc{X} \setminus \mc{X}^{(1)}$ into fewer than $2k/3$ clusters. 

For this, we begin by writing $\mc{B} := \{x\in\mb{R}^p: \|x\|_{2} \leq 1\}$ as a disjoint union of cubes, denoted by $\mf{C}$, each with side length $\gamma = e^{-\log{k}/(4p)}/\sqrt{p}$. By a standard volumetric estimate (see, e.g., \cite[Proposition~3.7]{BSV21}), the number of cubes in $\mf{C}$ is at most $k^{1/3}$ (if $c < 1/120$, say, and $k$ is sufficiently large). Therefore, it suffices to cluster the points in each cube into at most $(2/3)k^{2/3}$ clusters. We have two cases: 
\begin{itemize}
    \item Case I: $\mc{C} \in \mf{C}$ satisfies $\mb{P}[X \in \mc{C}\setminus \mc{X}^{(1)}] \leq k^{-1/2}$. Let us denote all such cubes by $\mf{C}_{1}$. In this case, we assign all the points in $\mc{C} \setminus \mc{X}^{(1)}$ to a single cluster (say, corresponding to the midpoint of $\mc{C}$).   
    \item Case II: $\mc{C} \in \mf{C}$ satisfies $\mb{P}[X \in \mc{C}\setminus \mc{X}^{(1)}] \geq k^{-1/2}$.  Let us denote all such cubes by $\mf{C}_{2}$. In this case, consider the random vector $X_{\mc{C}}$, which takes on each value $x \in \mc{C}\setminus \mc{X}^{(1)}$ with probability $\mb{P}[X=x]/\mb{P}[X \in \mc{C}\setminus \mc{X}^{(1)}]$. Note that $X_{\mc{C}}$ is supported on a $p$-dimensional cube of side length $\gamma$, and for any $x \in X_{\mc{C}}$, we have that $\mb{P}[X_{\mc{C}} = x] \leq (3/k)/k^{-1/2} \leq 3k^{-1/2} \leq k^{-1/3}$. We partition the points in $\mc{C}\setminus \mc{X}^{(1)}$ according to the clusters coming from \cref{prop:key-estimate} applied to $X_{\mc{C}}$, noting that there are at most $2^{p} < k^{1/2}$ clusters for each cube $\mc{C} \in \mf{C}_{2}$ (provided that $c < 1/2$). Denote the corresponding sigma algebra by $\mc{F}_{\mc{C}}$.  
\end{itemize}

At this point, we have partitioned the points in $\mc{X}$ into at most $k/3 + k^{1/3}\cdot k^{1/2} < k/2$ clusters. To complete the proof, we check that the sigma algebra $\mc{F}$ generated by this clustering satisfies the conclusion of \cref{thm:main}. Letting $Y = \mb{E}[X \mid \mc{F}]$, we have
\begin{align*}
    \|\Sigma_{X} - \Sigma_Y \|_{\mr{F}} 
    &= \|\mb{E}[XX^T] - \mb{E}[YY^T]\|_{\mr{F}}\\
    &\leq \sum_{\mc{C} \in \mf{C}_1}\mb{P}[X \in \mc{C}\setminus \mc{X}^{(1)}]\cdot \gamma^{2}p + \sum_{\mc{C} \in \mc{C}_2}\mb{P}[X \in \mc{C} \setminus \mc{X}^{(1)}]\cdot \|\Sigma_{X_{\mc{C}}} - \Sigma_{\mb{E}[X_{\mc{C}} \mid \mc{F}_{\mc{C}}]}\|_{\mr{F}} \\
    &\leq k^{1/3}\cdot k^{-1/2} + \sum_{C \in \mc{C}_{2}}\mb{P}[X \in \mc{C} \setminus \mc{X}^{(1)}]\cdot \left(\frac{36e^{-1/(2c)}}{\sqrt{c\log{k}}}+k^{-1/48}\right) \quad \quad \textrm{(\cref{prop:key-estimate})}\\
    &\leq \frac{40}{\sqrt{c\log{k}}},
\end{align*}
provided that $c < 1/120$ and $k$ is sufficiently large.

\bibliographystyle{amsplain0.bst}
\bibliography{main.bib}

\end{document}